\documentclass{amsart}
\usepackage[english]{babel}

\usepackage{amsfonts}
\usepackage{amssymb}
\usepackage{color}

\newtheorem{thm}{Theorem}[section]

\newtheorem{lem}[thm]{Lemma}

\theoremstyle{definition}

\theoremstyle{remark}

\numberwithin{equation}{section}

\begin{document}

\title[On Runge type theorems]{On approximation theorems for solutions to strongly
 parabolic systems in anisotropic Sobolev spaces }

\author{A.A. Shlapunov}
\address[Alexander Shlapunov]
{Siberian Federal University
                                                 \\
         pr. Svobodnyi 79
                                                 \\
         660041 Krasnoyarsk
                                                 \\
         Russia}
\email{ashlapunov@sfu-kras.ru} 

\author{P.Yu. Vilkov}
\address[Pavel Vilkov]
{Siberian Federal University
                                                 \\
         pr. Svobodnyi 79
                                                 \\
         660041 Krasnoyarsk
                                                 \\
         Russia}
\email{pavel\_vilkov17@mail.ru}

\subjclass {Primary 35A35; Secondary  35N17, 35K25}
\keywords{approximation theorems, anisotropic Sobolev spaces, strongly 
uniformly parabolic  
operators}

\begin{abstract}
We investigate the problem on Runge pairs  for Sobolev solutions of strongly uniformly parabolic
systems in non-cylindrical domains of a special kind. We prove that if the coefficients
of a parabolic operator are constant, then two domains with sufficiently smooth boundaries,
no parts of which are parallel to the plane $t=0$, form a Runge pair if and
only if the complements of any section of the larger domain to the section of the smaller domain
by planes $t = const$, have no compact components in the larger section.
\end{abstract}

\maketitle

\section*{Introduction}
\label{s.Int}

Approximation theorems for solutions to parabolic operators appeared since 1970-th, see 
\cite{D80}, \cite{J},  for the heat equation. Actually, these were generalizations of 
the realated theorems in Complex Analysis  see \cite{H68}, \cite{R1885}, \cite{V67} and 
\cite{Brw62}, \cite{Lax}, \cite{Mal56}, \cite{Mal63}, \cite{MaHa74}, \cite{Merg56}, \cite{Tark36}, 
\cite{Tark37} in the theory of (mostly, elliptic) PDE's. Recently, the interest to approximation theorems 
for general parabolic and elliptical parabolic   systems has essentially increased, see, for instance, 
\cite{EP-S15}, \cite{EG-FP-S19}, \cite{EP-S21}, \cite{G-FRZ}, \cite{GauTa10}, \cite{Klm21},
\cite{RueSa}, \cite{ShHeat}, \cite{VSh2024}, \cite{VKuSh2022}. The main reason for this interest is the following: 
the classical theory of boundary value problems for parabolic and elliptical parabolic   systems was 
developed in cylinder domains where the method of separating variables allowed to use approximation theorems 
for elliptic operators instead, see, for instance, \cite{eid}, \cite{LadSoUr67}, 
\cite{Tema79}. Nowdays, a possibility to consider the boundary problems parabolic systems 
in non-cylinder domains becomes very important.  For instance, it is the case in Cardiology, because the shape of heart changes over time, see, for instance, \cite{KSU_ZAMM} where an the ill-posed problem for a parabolic equation was used in the bi-domain model of the myocardium.

In the pioneer papers, Runge considered a pair of domains $D_1 \subset D_2$ of the complex plane and 
the problem of approximation of holomorphic functions in $D_1$ by holomorphic functions in $D_2$ (here the spaces were endowed with the natural Fr\'echet topology of the space $C(D)$). 
Various authors used different topologies in the approximation theorems (uniform topology of convergence 
on compact subsets, toplogies of Lebesgue spaces, Sobolev spaces, etc.). But it turned out that, for 
elliptic systems, conditions granting for two domains to be a Runge pair are depend essentially on 
the behaviour of the complement $D_2 \setminus D_1$. For the heat equation, domains $G_1 
\subset G_2 \subset {\mathbb R}^{n+1}$ form a Runge pair if and
only if the complements $G_{2} (c) \setminus G_{1} (c)$ of any section $G_{2} (c)$ 
 of the larger domain to the section $G_{1}(c)$ of the smaller domain by hyperplanes $t = c = const$, have no compact components in the larger section.  

Actually, as the Faedo-Galerkin method is still very efficient for solving boundary value problems, the consideration 
of Runge pairs allows to construct series representing a solution to a problem with summands regular at 
a greater domain. Besides, it is important for applications to have approximation theorems 
with respect to toplogies controlling the behaviour of functions up to the boundary of the related domain; see, 
for instance, Runge type approximation theorems in the Lebesgue space for analytic functions \cite{H68}, 
for strongly parabolic operators  \cite{VSh2024}, \cite{VKuSh2022}, and 
approximation theorems in the Sobolev spaces for elliptic operators \cite[Ch. 5]{Tark36}, \cite[Ch. 8]{Tark37}.

In this paper we concentrated our efforts on the Runge type approximation theorems 
for strongly $2m$-parabolic operators  with constant coefficients in anisotropic Sobolev spaces. The primary 
motivation is the systematic use of this type of spaces for solving boundary value problems. 
We rely essentially on the results of the paper \cite{VSh2024} on the Runge type approximation theorems 
for Lebesgue $L^2$-solutions to strongly $2m$-parabolic operators keeping the related denotions in the present paper.

\section{Preliminaries}

Let ${\mathbb R}^n$, $n \geq 1$, be the $n$-dimensional Euclidean space with the coordinates  
$x=(x_1, \dots , x_n)$ and let  $\Omega \subset {\mathbb R}^n$ be a bounded domain 
(open connected set). As usual, denote by  $\overline{\Omega}$ the closure of $\Omega$, and by
 $\partial\Omega$ its boundary.  
 
We consider functions over  ${\mathbb R}^n$ and 
${\mathbb R}^{n+1}$.  As usual, for $s \in {\mathbb Z}_+$ we denote by $C^s(\Omega)$ 
and $C^s(\overline \Omega)$  the spaces of all $s$ times continuously differentiable 
functions on $\Omega$ and $\overline \Omega$, respectively.  
The spaces $C^{s}(\overline \Omega)$ are known to be Banach spaces with the standard  
norms and the  $C(\Omega)$ are the Fr\'echet spaces with the standard semi-norms.

Let  also $L^2 (\Omega)$ be the Lebesgue space over $\Omega$ with the standard inner product 
$(u,v)_{L^2 (\Omega)} $ and let $H^s (\Omega)$, $s\in \mathbb N$, be the Sobolev space  
with the standard inner product $(u,v)_{H^s (\Omega)}$. As usual, we 
consider the Sobolev space $H^{-s} (\Omega)$, $s\in \mathbb N$, as the dual space 
of $H^{s}_0 (\Omega)$ where $H^{s}_0 (\Omega)$ is the closure of the space 
$C^\infty_{\rm comp} (\Omega)$ consisting of smooth 
functions with compact supports in $\Omega$.  

Given $m,s\in \mathbb  N$, 
we also need the anisotro\-pic Sobolev 
spaces $H^{2ms,s} (G)$, $s \in  {\mathbb Z}_+$, in a domain $G\subset {\mathbb R}^{n+1} $ with the standard inner product,
$$
(u,v)_{H^{2ms,s} (G)} = \sum_{|\alpha|+2mj\leq 2ms} 
(\partial^\alpha_x \partial^j_t  u, \partial^\alpha_x \partial^j_t v)_{L^2 (G) }.
$$ 
Besides, for $\gamma \in {\mathbb Z}_+$, we denote by $H^{\gamma,2sm,s} (G)$ 
the set of all functions  $u \in H^{2sm,s} (G)$ such that 
 $\partial ^\beta_x u \in H^{2ms,s} (G)$ for all  $|\beta|\leq \gamma$.  
It is convenient to denote by 
$\mathbf{H}^{2ms,s}_k (G)$ the space of all the $k$-vector functions with 
the components from ${H}^{2ms,s} (G)$, and similarly for the spaces
$\mathbf{L}^{2}_k (G)$, $\mathbf{H}^{\gamma, 2ms,s}_k (G)$, etc. These are known to be Hilbert spaces, see 
\cite{Kry08}. 

 Let  $L$ be a 
$(k\times k)$-matrix differential operator with constant coefficients in 
${\mathbb R}^{n}$ of an even order $2m$:
$$
L = \sum_{|\alpha|\leq 2m} L_{\alpha}  \partial ^\alpha_x 
$$ 
where $L_{\alpha} $ are $(k\times k)$-matrices with real entries 
such that $L^*_\alpha = L_\alpha $ for all multi-indexes $\alpha \in {\mathbb Z}_+^n$ 
with $|\alpha|=2m$. Consider the strongly 
uniformly (Petrovsky) $2m$-parabolic operator  
$$
{\mathcal L}= \partial _t - L,
$$  
see, for instance, \cite{eid}, \cite{sol}. More precisely, this 
additionally means that the operator $(-L)$ is strongly elliptic, i.e. 
there is a positive constant $c_0$ such that 
$$
(-1)^{m+1} w^* \Big(
\sum_{|\alpha| =2m}L_\alpha \zeta ^\alpha \Big) w \geq c_0 |w|^2 |\zeta|^{2mk}
$$
for all $\zeta \in {\mathbb R}^n \setminus 
\{0\}$ and all $w \in {\mathbb C}^k \setminus \{0\}$; here $w^*$ is the transposed and complex 
adjoint vector for the complex vector $w \in {\mathbb C}^k$. 

As usual, we denote by ${\mathcal L}^*$ the formal adjoint operator for ${\mathcal L}$:
$$
{\mathcal L}^*= - \partial _t -  \sum_{|\alpha|\leq 2m} (-1)^{|\alpha|} \partial ^\alpha_x ( 
L^*_{\alpha} (\, \cdot). 
$$

Under these assumptions  the operator ${\mathcal L}$ admits a unique fundamental 
solution  $\Phi (x-y,t-\tau)$ of the convolution type, possessing standard estimates \cite[formulas (2.16), 
(2.17)]{eid}) and the normality property (\cite[Property 2.2]{eid}), i.e. 
\begin{equation} \label{eq.right}
{\mathcal L}_{x,t} \Phi(x-y,t-\tau) = I_k \, \delta (x-y, t-\tau),   
\end{equation}
(here the right hand side equals to the unit matrix $I_k$ multiplied by the Dirac 
distribution at the point $(x,t)$ which is commonly written as $\delta (x-y, t-\tau)$, 
where $\delta $ denotes the Dirac distribution at the origin), and   
\begin{equation} \label{eq.left}
{\mathcal L}^*_{y,\tau} \Phi^*(x-y,t-\tau)  =I_k \delta (x-y, t-\tau), 
\end{equation}
where $\Phi^* = (\Phi_{ji})$ is the adjoint matrix for $\Phi =(\Phi_{ij})$. 

Let $S _{\mathcal L}(G)$ be the set of all the generalized 
$k$-vector functions on $G$, satisfying the (homogeneous)  equation
\begin{equation} \label{eq.heat}
{\mathcal L} u = 0 \mbox{ in } G 
\end{equation}
in the sense of distributions. 
Also, let the space  $S _{\mathcal L}(\overline{G})$ be 
defined as follows: 
$$
\cup_{U \supset 
\overline{G} } S _{\mathcal L}(U),
$$
where the union is with respect to all the domains $U \subset {\mathbb R}^{n+1} $,  containing the closure of the domain  $G$.

Then estimates  
\cite[formulas (2.16), (2.17)]{eid}) for the fundamental solution imply the standard interior 
a priori estimates for solutions to \eqref{eq.heat}, see, for instance, 
\cite[\S 19]{sol}, or \cite[Ch. 4, \S 2]{frid} for the second order operators. This means 
that the operator ${\mathcal L}$ is hypoelliptic, i.e. all the distributional solutions to equation 
\eqref{eq.heat} are $C^\infty$-differentiable on 
their domain. Then the following embeddings hold true:
$$
S _{\mathcal L}(\overline G) \subset S _{\mathcal L}(G) \subset \mathbf{C}^{\infty} _k (G).
$$ 
Moreover, as the coefficients of ${\mathcal L}$ are constant,
 the elements of the spaces $S_{\mathcal L}(G) $, $S _{\mathcal L}(\overline G)$ are  real 
analytic with respect to the space variable $x\in G(t)$ for all $t \in (T_1,T_2)$, 
where $ T_1=\inf_{(x,t)\in G} t$, $ T_2=\sup_{(x,t)\in G} t$, and 
$$
G(t) = \{x \in {\mathbb R}^n: (x,t) \in G\},
$$ 
see, for 
instance \cite{eid}. In particular, the so-called Unique Continuation Property with respect 
to the space variables $x$ for each fixed $t$ holds true for both the parabolic operator $\mathcal L$ and 
the backwards parabolic operator ${\mathcal L}^*$.

Given a pair of domains $G_1\subset G_2 \subset {\mathbb R}^{n+1}$, Runge type approximation theorems 
were proved in \cite{VSh2024} for the Fr\'echet spaces $S _{\mathcal L}(G_1) \supset S _{\mathcal L}(G_2)$ endowed 
with the  topology of the uniform convergence on compact subsets of $G_j$ and the Hilbert spaces 
$$ \mathbf{L}^{2} _{k,\mathcal L}(G_1) =  S _{\mathcal L}(G_1) \cap \mathbf{L}^{2}_k (G_1) \supset 
\, \mathbf{L}^{2} _{k,\mathcal L}(G_2) =  S _{\mathcal L}(G_1) \cap \mathbf{L}^{2}_k (G_2) .
$$

Let us formulate the related result for the spaces 
$ \mathbf{L}^{2} _{k,\mathcal L}(G_j)$ in our particular situation of operators with 
constant  coefficients. With this purpose we need additional  regularity assumptions
for the domains $G_1$ and $G_2$. 
Namely, we assume that the boundary 
of domains $G_j$ satisfies the following properties.

\begin{itemize}
\item[(A)] 
The set $G_2(t)$ is a Lipshitz domain in ${\mathbb R}^n$ for each $t \in (T_1,T_2)$ and for any numbers 
$t_3, t_4$ such that $T_1<t_3<t_4<T_2$ the set $\Gamma _{t_3,t_4}=\cup_{t\in [t_3,t_4]} \partial G (t)$ 
is a Lipschitz surface in ${\mathbb R}^{n+1}$;
\end{itemize}

\begin{itemize}
\item[(A1)]    
For each $t\in [T_1, T_2]$, the sets  
$G_1 (t) = \{x \in {\mathbb R}^n: (x,t) \in G_1\}$, 
are  domains in ${\mathbb R}^n$ with $C^{2m}$-boundaries if $n\geq 2$; 
\item[(A2)] 
The boundary $\partial G_1$ of $G_1$ is the union $G (T_1)\cup  G (T_2) \cup \Gamma $, where 
$$\Gamma =\cup_{t\in (T_1,T_2)} \partial G_1 (t)$$ 
is a $C^{2m,1}$-smooth surface without  points where the tangential planes are parallel to the 
coordinate plane $\{t=0\}$, i.e. we have 
$$
\sum_{j=1}^n  (\nu_j (x,t))^2 \geq \varepsilon _0 \mbox{ for all } (x,t) \in \Gamma
$$
with a positive number $\varepsilon _0$.
\end{itemize}

Of course, a cylinder domain $G = \Omega \times (T_1, T_2)$ 
satisfies assumption (A) if  $\partial \Omega$ is a Lipschitz surface and it 
satisfies assumptions (A1), (A2) if  $\partial \Omega$ is a $C^{2m}$-smooth surface. 

\begin{thm}
\label{t.dense.base.const}
Let  $G_1 \subset G_2 $   be domains 
in  ${\mathbb R}^{n+1}$ such that 
$G_2 \ne {\mathbb R}^{n+1} $. 
If    domain $G_2$ 
satisfies assumption $\mathrm{(A)}$ and bounded domain $G_1$  satisfies $\mathrm{(A1)}$, $\mathrm{(A2)}$ 
then $S_{\mathcal L}(\overline G_2)$ is everywhere dense in the space $\mathbf{L}^{2} _{k,\mathcal L}(G_1)$ 
if and only for each $t \in  (T_1,T_2)$  
the set  $G_2(t)\setminus G_1 (t) $ has no compact (non-empty) components in the set $G_2 (t)$. 
 \end{thm}

\begin{proof} As the Unique Continuation Property with respect 
to the space variables $x$ for each fixed $t$ holds true for both  $\mathcal L$ and 
 ${\mathcal L}^*$, the statement follows from \cite[Theorems 2.2, 2.4]{VSh2024}.
\end{proof}

A similar theorem for the heat equation was obtained in \cite{ShHeat} for cylinder domains and 
in \cite{VKuSh2022} a similar result was proved for the strongly parbolic Lam\'e type system. 
Of course, for general strongly $2m$-parabolic operators with variable coefficients the situation is 
more complicated, but for operators with bounded real analytic coefficients the answer is practically the same, 
see \cite{VSh2024}. 

We want to extend Theorem \ref{t.dense.base.const} to the scale of anisotropic spaces 
$\mathbf{H}^{\gamma, 2ms,s} _{k,\mathcal L}(G_j)$, $1\leq j \leq 2$, $s \in \mathbb N$,
$$
\mathbf{H}^{\gamma, 2ms,s} _{k,\mathcal L}(G) = \mathbf{H}^{\gamma,2ms,s}_k (G) \cap S_{\mathcal L}(G). 
$$
Similarly to the space $\mathbf{L}^{2} _{k,\mathcal L}(G)$, 
by a priori estimates for strongly parabolic operators, the spaces 
$\mathbf{H}^{\gamma, 2ms,s} _{k,\mathcal L}(G)$ are closed subspaces of 
the Hilbert spaces $\mathbf{H}^{\gamma, 2ms,s} _{k}(G)$.

Next, denote by $\tilde{\mathbf{H}}^{-\gamma, -2ms,-s} _{k}(G)$ 
the completion of $\mathbf{C}^{\infty}_k  (\overline G)$ with respect to the norm
$$
\| v\|_{\tilde{\mathbf{H}}^{-\gamma, -2ms,-s} _{k}(G)} = 
 \sup_{ 
\|\varphi\|_{ {\mathbf{H}^{\gamma, 2ms,s} _{k}(G) \leq 1} \atop {\varphi \in \mathbf{H}^{\gamma, 2ms,s} _{k}(G)}
}}
 \Big|(v,\varphi)_{{\mathbf L}^2_k (G)}\Big| , \,\, s \in {\mathbb N};
$$
here $\tilde H$ reflects the fact the usually 
$H^{-s} (G)$ is reserved for the dual space to  $H^{s}_0 (G)$ with slightly different norm. 
As it is known, the space $\tilde{\mathbf{H}}^{-\gamma, -2ms,-s} _{k}(G)$ is a Banach space, see \cite{Adams}. 
The following property is important for the futher exposition. 

\begin{lem} \label{l.appr} 
Let $G$ be a bounded domain in ${\mathbb R}^{n+1}$.  
Then the space $\mathbf{C}^{\infty}_{0,k} (G)$  is everywhere dense in 
$\tilde{\mathbf{H}}^{-\gamma, -2ms,-s} _{k}(G)$.
\end{lem}

\begin{proof} Indeed, by the very definition, the space $\mathbf{C}^{\infty}_{k} (\overline G)$  is everywhere 
dense in $\tilde{\mathbf{H}}^{-\gamma, -2ms,-s} _{k}(G)$. On the other hand, as it is well known that 
$\mathbf{C}^{\infty}_{0,k} (\overline G)$ is everywhere dense in the space $\mathbf{L}^{2} _{k}(G)$. 
Now we note that the norm $\|\cdot\|_{\mathbf{L}^{2} _{k}(G)}$ is stronger than the norm 
$\|\cdot\|_{\tilde{\mathbf{H}}^{-\gamma, -2ms,-s} _{k}(G)}$ on $\mathbf{C}^{\infty}_{k} (\overline G)$ for 
any $s \in \mathbb N$. Thus, the statements of this lemma follows because we may approximate any element of 
the space $\mathbf{C}^{\infty}_{k} (\overline G)$ by $\mathbf{C}^{\infty}_{0,k} ( G)$-vectors 
in $\|\cdot\|_{\mathbf{L}^{2} _{k}(G)}$-norm. 
\end{proof}

\section{Approximation in the anisotropic Sobolev spaces}

The main result of this paper is the following theorem.

\begin{thm}
\label{t.dense.Sob.const}
Let  $s \in \mathbb N$, $\gamma \in {\mathbb Z}_+$ $G_1 \subset G_2 $   be domains 
in ${\mathbb R}^{n+1}$ such that 
$G_2 \ne {\mathbb R}^{n+1} $. 
If    domain $G_2$ 
satisfies assumption $\mathrm{(A)}$ and bounded domain $G_1$  satisfies $\mathrm{(A1)}$, $\mathrm{(A2)}$ 
then $S_{\mathcal L}(\overline G_2)$ is everywhere dense in the space $\mathbf{H}^{\gamma,2ms,s} _{k,\mathcal L}(G_1)$ 
if and only for each $t \in  (T_1,T_2)$ 
the set  $G_2(t)\setminus G_1 (t) $ has no compact (non-empty) components in the set $G_2 (t)$. 
 \end{thm}

\begin{proof} We begin with the necessity. To prove it  we may use the arguments similar to the proof of 
\cite[Theorem 1.2]{VSh2024} adapting them for the present situation. 

Indeed, as the boundary of the domain $G_1$ is at least $C^1$-smooth
then the complement $G_2\setminus \overline G_1$ is an open set with Lipschitz boundary. 
If there is a number $t_0\in \mathbb R$ such that the set 
$G_2 (t_0) \setminus G_1 (t_0) $  has a compact non-empty (connected) component $K (t_0)$ in the set 
$G_2 (t_0)$ then $K(t_0)$ is a closure of a  domain $D_0 \subset G_2 (t_0)$ with $C^1$-smooth boundary.   
We have to prove that there is a vector $u\in {\mathbf H}^{2ms,s}_{k,\mathcal L} (G_1)$ 
that can not be approximated by elements of $S_{\mathcal L} (\overline G_2)$. 
Now, if a point $y_0 \in K_0$, then  $(y_0,t_0)$ is an interior point of $G_2\setminus \overline G_1$ an hence 
any vector column $U_l (x,t)$, $1\leq l \leq k$, 
of the fundamental matrix $\Phi  (x-y_0, t-t_0)$ 
belongs to the space ${\mathbf H}^{\gamma,2ms,s}_{k,{\mathcal L}} (G_1)$.

Now we may envoke the (second) Green formula for the operator ${\mathcal L}$, see, for instance, 
\cite{Tark35} for general operators admitting regualar left fundamental solutions or \cite{svesh} 
for the second order parabolic operators.
Namely, let ${\mathcal G}_{ L}$  be  a Green bi-differential operator for the operator $L$, 
\cite[\S 2.4.2]{Tark35}. As it is known, 
it has order $(2m-1)$ with respect to the space variables $x$, acting from 
$ \mathbf{C}^{2m,1}_k (\overline G_2) \times \mathbf{C}^{2m,1}_k (\overline G_2) $
to the space of $n$-differential forms with  coefficients from $C^1 (\overline G_2)$, i.e. 
\begin{equation}\label{eq.Green.L}
\int_{\partial G_3 }
 {\mathcal G}_{ L} (g,v) = (Lv,g)_{\mathbf{L}^2 _k(G_3)} -
( v,{L} ^* g)_{\mathbf{L}^2 _k(G_3)} \mbox{ for all } g,v\in  \mathbf{C}^{2m,1}_k (\overline G_3) 
\end{equation}
 and any domain $G_3\Subset G_2$ with piecewise smooth boundary. 
As the Green operator of a sum of differential operators can be presented as the sum of the corrresponding 
Green operators, then for ${\mathcal L}$ we obtain: 
$$
 {\mathcal G}_{ \mathcal L} (g,v) = g^* v dx -  {\mathcal G}_{ L} (g,v),
$$
and   then the (first) Green formula holds true:
\begin{equation}\label{eq.Green.parab}
\int_{\partial G_3 }
 {\mathcal G}_{ \mathcal L} (g,v) = ({\mathcal L}v,g)_{\mathbf{L}^2 _k(G_3)} - 
( v,{\mathcal L} ^* g)_{\mathbf{L}^2 _k(G_3)} \mbox{ for all } g,v\in  \mathbf{C}^{2m,1}_k (\overline G_3). 
\end{equation}

Since $(y_0,t_0)$ is an interior point of $G_2\setminus \overline G_1$, 
we may choose a bounded domain  $G_3$ with a piecewise smooth boundary $\partial G_3$ such 
that $(y_0,t_0) \in G_3 \Subset G_2$ and $\partial G_3 \Subset G_1$.  If the vector function 
$U_l (x,t)$ can be approximated in ${\mathbf H}^{\gamma, 2ms,s}_k  (G_1)$, $s\in \mathbb N$,  
by a sequence $\{ u^{(i)}_l \}_{i\in \mathbb N}$ from the 
space $S_{\mathcal L} (\overline G_2)$ 
then  the sequences of the partial derivatives  
$\{ \partial ^\beta_x \partial ^\alpha_x \partial_t^j u^{(i)}_l \}$, $|\alpha|+2mj\leq 2ms$, $|\beta|\leq \gamma$, 
converge uniformly on $\partial G_3$. On the other hand,  (the first) Green formula
\eqref{eq.Green.parab} and the normality property \eqref{eq.left} of the fundamental solution $\Phi$ 
imply the (second) Green formulas for all $i\in \mathbb N$: 
\begin{equation}\label{eq.Green.parab.2}
u^{(i)}_l (x,t) = - \int_{\partial G_3 }
 {\mathcal G}_{ \mathcal L} (\Phi (x-y,t-\tau), u^{(i)}_l (y,\tau)) \mbox{ for all } (x,t) \in G_3.
\end{equation}
Note that there is no need to assume that $G_3$ is a cylinder domain because this Green formula 
is a corollary of the \textit{local} reproducing property of the fundamental solution. 
Now, passing to the limit with respect to $i\to +\infty$ in \eqref{eq.Green.parab.2} we obtain 
\begin{equation}\label{eq.Green.parab.3}
U_l (x,t) = - \int_{\partial G_3 }
 {\mathcal G}_{ \mathcal L} (\Phi (x-y,t-\tau), U_l (y,\tau)) \mbox{ for all } (x,t) \in G_3 \cap G_1.
\end{equation}
However, since $\Phi$ is a fundamental solution to ${\mathcal L}$ then the right-hand side of formula 
\eqref{eq.Green.parab.3} belongs to $S_{\mathcal L} (G_3)$. Therefore 
the vector function $U_l $ extends as a solution $V_l$ to equation \eqref{eq.heat} from $G_1\cap G_3 $ 
to $G_3$, i.e.  to a neighbourhood of the point $(y_0,t_0)$. In particular, since any solution to 
the operator ${\mathcal L}$ in $G_2$ is real analytic 
with respect to the space variables in $G_2(t) $ for each $t \in (T_1, T_2)$, then 
 this extension is unique on $G_3\setminus (y_0,t_0)$.
This means the vector function $V_l \in S_{\mathcal L} (G_3)$
coincides with the $l$-th vector column $U_l$ of the fundamental matrix $\Phi (x-y_0,t-t_0)$ in 
$G_3 \setminus (y_0,t_0))$. Thus, we obtain a contradiction because for the matrix $V (x,t)$ with columns $V_l$, $1\leq l \leq k$  we have ${\mathcal L}V  =0$ in $G_3$  
  but ${\mathcal L} \Phi  (x,y_0,t,t_0)$ coincides with the $\delta$-functional 
concentrated at the point $(y_0,t_0)$ multiplied on $(k\times k)$-unit matrix. The necessity is proved.  

Next, we proceed with the sufficiency. 
Actually we slightly modify  the proof from \cite{D80} for the solutions  
to the heat equation, adapting it to the topology  of the anisotropic  Sobolev spaces, cf. 
\cite{ShHeat}, \cite{VSh2024} for the approximation in the Lebesgue spaces. 

Indeed, let for each $t \in  (T_1,T_2)$ 
the set  $G_2(t)\setminus G_1 (t) $ have no compact (non-empty) components in the set $G_2 (t)$. 
The Hahn-Banach Theorem implies that $G_1$, $G_2$ is a 
${\mathcal L}$-Runge's pair in the sense of the present theorem 
if and only if any continuous linear functional $f$ on $\mathbf{H}^{\gamma, 2ms,s} _{k,\mathcal L}(G_1)$ annihilating the space $S_{\mathcal L}(\overline G_2)$ also annihilates the space 
$\mathbf{H}^{\gamma, 2ms,s} _{k,\mathcal L}(G_2)$.

As we have noted above, the space $\mathbf{H}^{\gamma, 2ms,s} _{k,\mathcal L}(G)$ is a closed subspace of the space 
$\mathbf{H}^{\gamma, 2ms,s} _{k}(G)$. Then any 
continuous linear functional $f$  on $\mathbf{H}^{\gamma, 2ms,s} _{k,\mathcal L}(G_1)$
can be extended as a continuous linear functional $F$ on $\mathbf{H}^{\gamma, 2ms,s} _{k}(G_1)$, i.e. as an 
element of the Sobolev space $\tilde {\mathbf{H}}^{-\gamma, -2ms,-s} _{k}(G_1)$.
In particular, 
$F$ can be identified as a distrubition $\psi$ on ${\mathbb R}^{n+1}$ supported in $\overline G_1$ and 
having a finite order of singularity:
\begin{equation} \label{eq.Riesz}
F(u) = \langle u , \psi  
\rangle \mbox{ for all } u \in \mathbf{H}^{\gamma, 2ms,s} _{k}(G_1).
\end{equation}
Let $W$ be the vector function, with components obtained by applying the functional $F$ to the corresponding columns of the 
matrix $(x,t) \to \Phi (x-y,t-\tau)$:  
\begin{equation*}
W(y,\tau)= \langle \Phi^* (x,y, t,\tau) , \psi (x,t) \rangle .
\end{equation*}
Clearly, it is well-defined outside the support of $\psi$. By  \eqref{eq.right}, for any vector 
$\varphi \in {\mathbf C}^\infty_{0,k} ({\mathbb R}^{n+1})$ we have  
$$
{\mathcal L}_{x,t} \, \langle \Phi^* (x,y, t,\tau) , \varphi (y,\tau) \rangle = \varphi (x,t),
$$
and then, the hypoellipticity of ${\mathcal L}$ implies that the vector function 
$$
V_\varphi(x,t)  = \langle \Phi^* (x,y, t,\tau) , \varphi (y,\tau) \rangle $$ 
belongs to 
$\mathbf{C}^{\infty}_k ({\mathbb R}^{n+1})$. In particular, the vector 
function $W$ can be extended as a distribution to ${\mathbb R}^{n+1} $ via 
$$
\langle W, \varphi \rangle = \langle V_\varphi, \psi \rangle, 
$$
because $\psi$ is a ($k$-vector valued) distribution with compact support. 

Next, according to \eqref{eq.right}, columns of the matrix $\Phi (x,t,y,\tau)$ belong to 
$S_{\mathcal L} (\overline G_2)$ with respect to variables $(x,t)$ for each fixed $(y,\tau) \not \in
\overline  G_2$.
Hence, 
if $F\in (\mathbf{H}^{\gamma, 2ms,s} _{k}(G_1))^*$ annihilates the space $S_{\mathcal L}(\overline G_2)$ then we have  
\begin{equation}
\label{eq.v1.A}
W(y,\tau)=  0 \mbox{ for all } 
(y,\tau) \not \in \overline G_2.
\end{equation}
But \eqref{eq.left} implies that 
\begin{equation} \label{eq.L*W.1}
{\mathcal L}^* W = \psi \mbox{ in } {\mathbb R}^{n+1} ,  
\end{equation}
in the sense of distributions and, in particular, 
\begin{equation} \label{eq.L*W.2}
{\mathcal L}^* W = 0 \mbox{ in } 
{\mathbb R}^{n+1}  \setminus \overline G_1.
\end{equation}
Note that the operator ${\mathcal L}^*$ is backwards-parabolic and, for any 
solution $v(y,\tau)$ to the equation ${\mathcal L}^* v =0$, the vector 
$w(y,\tau)=v(y,-\tau)$  is a solution to the strongly parabolic system of equations  
$(\partial_\tau - L^* _y) w =0$ with constant coefficients. 
Thus, by the hypoellipticity of such systems, $W(y,\tau) \in \mathbf{C}^{\infty}_k
 \big({\mathbb R}^{n+1} \setminus \overline G_1 \big)
$ and, 
in particular, it is $C^{\infty}$-smooth with respect to  $y$ in ${\mathbb R}^n 
\setminus \overline G_1 (\tau)$ for each $\tau \in \mathbb R$ where, as before,  
$\overline G_1 (\tau)= \{x\in {\mathbb R}^n: (x,\tau) \in \overline G_1\}$.

As the domains $G_1\subset G_2$ satisfy assumptions $\mathrm{(A)}$, 
$\mathrm{(A_1)}$, $\mathrm{(A_2)}$ and $G_2 \ne {\mathbb R}^{n+1} $, the components of sets 
${\mathbb R}^n\setminus   G_2 (t) \subset {\mathbb R}^n\setminus G_1 (t)$ are either empty sets 
or closures of Lipschitz domains. 
Since the  set $G_2 (t) \setminus G_1 (t)$ has no compact components in $G_2 (t)$, we see that 
 each bounded component of  ${\mathbb R}^{n} \setminus \overline {G_1 (t)} $ intersects with
 ${\mathbb R}^{n} \setminus  \overline {G_2 (t)} $ by a non-empty open set for each $t\in (T_1,T_2)$. 
Hence, as solutions to the backwards parabolic operator ${\mathcal L}^*$ are real analytic 
with respect to the space variables $x$, the vector $W$ vanishes on every bounded component 
of ${\mathbb R}^{n} \setminus \overline {G_1 (t)} $ for each $t\in (T_1,T_2) $.
Next, let $\hat G _j (t) $ be the union of $G _j (t)$ with all the 
components of the set $G _j (t)$ that are relatively compact in 
${\mathbb R}^n$. By the discussion above, the closure of $\hat G _1 (t)$ lies in the closure 
of $\hat G _2 (t) $. Then, by De Morgan’s Law 
we have 
\begin{equation} \label{eq.DeM}
\Big( {\mathbb R}^n \setminus \overline{\hat G _1 (t) }\Big)\cap 
\Big( {\mathbb R}^n \setminus \overline{\hat G _2 (t) }\Big) =   
{\mathbb R}^n \setminus \Big( \overline{\hat G _2 (t) } \cup \overline{\hat G _1 (t) }  \Big)= 
{\mathbb R}^n \setminus  \overline{\hat G _2 (t) }.
\end{equation}
In particular, this means that the vector $W$ vanishes on unbounded components of the set 
${\mathbb R}^n \setminus \overline {G _1 (t)} $ for each $t\in (T_1 (G_1), T_2 (G_1))$, too. 
Thus, \eqref{eq.v1.A} and the real analyticity  
with respect to the space variables of solutions to  
the operator ${\mathcal L}^*$ imply that 
\begin{equation*} 
W (y,\tau) = 0   \mbox{ in }{\mathbb R}^{n} \setminus \overline{G_1 (\tau)}
\mbox{ for all } \tau\in {\mathbb R}, 
\end{equation*}
i.e. the vector $W$  is supported in $\overline G_1$. 

Using Lemma \ref{l.appr}, we approximate the distribultion $\psi \in 
\tilde {\mathbf{H}}^{-\gamma, -2ms,-s} _{k}(G_1)$ by a sequence 
$\{ \psi_i \} \subset \mathbf{C}^\infty _{0,k} (G_1)$ in the space 
$\tilde {\mathbf{H}}^{-\gamma, -2ms,-s} _{k}(G_1)$. Then for vectors 
\begin{equation*}
W_i(y,\tau)= \langle \Phi^* (x,y, t,\tau) , \psi_i (x,t) \rangle 
\end{equation*}
we have 
\begin{equation} \label{eq.L*Wi.1}
{\mathcal L}^* W_i = \psi_i \mbox{ in } {\mathbb R}^{n+1} ,  
\end{equation}
and, in partucilar, by the hypoellipticity of parabolic operators with constant 
coefficients, $\{ W_i\} \subset \mathbf{C}^\infty _k ({\mathbb R}^{n+1})$. Since 
$\mathbf{C}^\infty _k ({\mathbb R}^{n+1}) \subset \tilde {\mathbf{H}}^{-\gamma, 2m(1-s),1-s} _{k}(G_1)$, 
Lemma \ref{l.appr} provides that 
we may approximate each vector $W_i$ by a sequence $\{ V_{i,j} \} \subset 
\mathbf{C}^\infty _{0,k} (G_1)$ in the space $\tilde {\mathbf{H}}^{-\gamma, 2m(1-s),1-s} _{k}(G_1)$.

Again using Lemma \ref{l.appr} we see that 

\begin{equation} \label{eq.est.norm.1}
\|{\mathcal L}^* (V_{i,j} - W_i)\|_{\tilde {\mathbf{H}}^{-\gamma, -2ms,-s} _{k}(G_1)} = 
 \sup_{\|\varphi\|_{
{\mathbf{H}^{\gamma, 2ms,s} _{k}(G_1)}\leq 1} \atop {\varphi \in \mathbf{C}^\infty _{0,k} (G_1)} }
\Big|({\mathcal L}^* (V_{i,j} - W_i),\varphi)_{{\mathbf L}^2_k (G)}\Big| =
\end{equation} 
\begin{equation*} 
\sup_{{\|\varphi\|_{\mathbf{H}^{\gamma, 2ms,s} _{k}(G_1)}\leq 1} \atop {\varphi \in \mathbf{C}^\infty _{0,k} (G_1)} }
\Big|(V_{i,j} - W_i, {\mathcal L} \varphi)_{{\mathbf L}^2_k (G)}\Big| =
\end{equation*} 
\begin{equation*} 
\sup_{\|\varphi\|_{{\mathbf{H}
^{\gamma, 2ms,s} _{k}(G)}\leq 1} \atop {\varphi \in \mathbf{C}^\infty _{0,k} (G_1) ,
{\mathcal L} \varphi \ne 0 }} \frac{\Big|
(V_{i,j} - W_i, {\mathcal L} \varphi)_{{\mathbf L}^2_k (G_1)}\Big|}{\|{\mathcal L} \varphi\|
_{\mathbf{H}^{\gamma, 2m(s-1),s-1} _{k}(G_1)}} \|{\mathcal L} \varphi\|
_{\mathbf{H}^{\gamma, 2m(s-1),s-1} _{k}(G_1)}.
\end{equation*} 
On the other hand, the scale of spaces $\mathbf{H}^{\gamma, 2ms,s} _{k}(G)$ is constructed in such way, 
that any $2m$-parabolic operator ${\mathcal L}$ continuously maps  $\mathbf{H}^{\gamma, 2ms,s} _{k}(G)$ to
$\mathbf{H}^{\gamma, 2m(s-1),s-1} _{k}(G)$, $s \in \mathbb N$, i.e. there exists a positive constant 
$ C(s,\gamma, {\mathcal L},G)$ such that
\begin{equation} \label{eq.est.norm.2}
\|{\mathcal L} v\|_{\mathbf{H}^{\gamma, 2m(s-1),s-1} _{k}(G)} \leq C(s,\gamma, {\mathcal L},G) \, 
\| v\|_{\mathbf{H}^{\gamma, 2ms,s} _{k}(G)} \mbox{ for all }  v \in \mathbf{H}^{\gamma, 2ms,s} _{k}(G).
\end{equation}
Then, taking into the account \eqref{eq.est.norm.1} and \eqref{eq.est.norm.2}, we conclude that 
\begin{equation} \label{eq.est.norm.3}
\|{\mathcal L}^* (V_{i,j} - W_i)\|_{\tilde {\mathbf{H}}^{-\gamma, -2ms,-s} _{k}(G_1)} \leq  
C(s,\gamma, {\mathcal L},G_1)
\|V_{i,j} - W_i\|_{\tilde {\mathbf{H}}^{-\gamma, 2m(1-s),1-s} _{k}(G_1)}, 
\end{equation}
i.e. the sequence $\{ {\mathcal L}^* V_{i,j}\} $ converges to ${\mathcal L}^* W_{i}$ in the space 
$\tilde {\mathbf{H}}^{-\gamma, -2ms,-s} _{k}(G_1)$ as $j \to +\infty$.

Hence, it follows from \eqref{eq.Riesz},  \eqref{eq.L*Wi.1}, \eqref{eq.est.norm.3} and 
the contunity of the functional $F$ that for all 
$u\in {\mathbf H}^{\gamma, 2ms,s}_{k,\mathcal L}(G_1)$ we have
\begin{equation} \label{eq.supported}
F (u) = \langle u ,\psi  \rangle = \lim_{i\to \infty}\langle u ,\psi_i  \rangle =  
 \lim_{i\to \infty} \langle  u , {\mathcal L}^* W_i \rangle =
\end{equation} 
\begin{equation*} 
\lim_{i\to \infty} \lim_{j\to \infty} \langle  u , {\mathcal L}^* V_{i,j} \rangle = 
\lim_{i\to \infty} \lim_{j\to \infty} \langle  {\mathcal L} u ,  V_{i,j} 
\rangle = 0,
\end{equation*} 
because ${\mathcal L} u =0 $ in $G_1$ in the sense of distributions. 
Thus, $F$ annihilates ${\mathbf H}^{\gamma, 2ms,s}_{k,\mathcal L}(G_1)$, too, that was to be proved. 
\end{proof}

\end{document}